\documentclass[leqno, 12pt]{article}
\usepackage{amsmath,amsfonts,amsthm,amssymb,indentfirst}
\usepackage{xy} \xyoption{all}
\usepackage[colorlinks]{hyperref}

\hypersetup{linkcolor=blue, urlcolor=blue, citecolor=red}

\newtheorem{theorem}{Theorem}[section]

\newtheorem{corollary}[theorem]{Corollary}

\newtheorem{lemma}[theorem]{Lemma}

\newtheorem{proposition}[theorem]{Proposition}

\theoremstyle{definition}
\newtheorem{definition}[theorem]{Definition}
\newtheorem{notation}[theorem]{Notation}
\newtheorem{example}[theorem]{Example}
\newtheorem{remark}[theorem]{Remark}

\begin{document}

\title{Products of Ideals in Leavitt Path Algebras}

\author{Gene Abrams, Zachary Mesyan, and Kulumani M.\ Rangaswamy}
\maketitle

\begin{abstract}
Ideals in Leavitt path algebras share many properties with those of integral domains. Since studying factorizations of ideals in integral domains into special types of ideals (e.g., prime, prime-power, primary, irreducible, semiprime, quasi-primary) has proved fruitful, we conduct an analogous investigation for Leavitt path algebras. Specifically, we classify the proper ideals in these rings that admit factorizations into products of each of the above types. We also classify the Leavitt path algebras where every proper ideal admits a factorization of these sorts, as well as those Leavitt path algebras where every proper ideal is of one of those types.

\medskip




\end{abstract}

\section{Introduction}\label{Introsection}

Since their introduction in 2004, Leavitt path algebras have become an active area of research. Although Leavitt path algebras $L$ are in general noncommutative, and as well possess plenty of zero divisors, recent investigations show that the (two-sided) ideals of such  algebras share many properties with the ideals of commutative integral domains. For example, multiplication of ideals in $L$ is commutative (\cite[Corollary 2.8.17]{AAS} and \cite[Theorem 3.4]{R2});  $L$ is  B\'{e}zout, that is,  all the finitely generated ideals of $L$ are principal \cite[Corollary 8]{R1}; $L$ is hereditary (i.e., ideals are projective as $L$-modules), a property of Dedekind domains \cite[Theorem 3.7]{AG}; and the ideal lattice of $L$ is distributive \cite[Theorem 4.3]{R2}, a characterizing property of Pr\"{u}fer domains among integral domains. (We note that the B\'{e}zout and hereditary properties hold for all one-sided ideals as well.)   It is well-known (see, e.g., \cite{F0}, \cite{MD}, \cite{OR}) that these integral domains admit satisfactory factorizations of their ideals as products and intersections of special types of ideals such as prime, prime-power, primary, irreducible, semiprime, or quasi-primary. In light of the aforementioned similarities, it is natural to investigate factorizations of the ideals of a Leavitt path algebra $L$ as products of these special types of ideals of $L$.

Although many different properties of ideals in general rings will be highlighted in this article, we distill our focus to the following two: primeness and semiprimeness. 
As we shall see, the other properties mentioned above, which in general rings are distinct from these, turn out to be subsumed by them in the context of ideal factorizations in Leavitt path algebras.  

In Section \ref{prelim-section} we present the standard graph-theoretic conditions which will play a role in the article.  We then give the definition of Leavitt path algebras, and present various well-known properties thereof, especially related to the ideal structure.    Subsequently, in the brief Section \ref{MultCondSection}, we collect various multiplicative conditions on ideals in general rings, and show that many of them are equivalent in the context of Leavitt path algebras (Proposition \ref{quasi-primaryequivalent}).

In Section \ref{primesection} we analyze the prime ideals of a Leavitt path algebra, relying on earlier results from the literature. Recall that a proper ideal $P$ in a not-necessarily-commutative ring $R$ is called $\textit{prime}$ in case whenever $I,J$ are ideals of $R$ for which $IJ \subseteq P$, then $I\subseteq P$ or $J\subseteq P$. Clearly this generalizes the standard definition of primeness in the commutative case.  This is equivalent to the condition that whenever $a,b\in R$ with $aRb \subseteq P$, then $a\in P$ or $b\in P$. We begin by classifying those Leavitt path algebras in which every proper ideal is prime (Theorem \ref{Everyone prime}). We then describe ideals which can be factored as products of (finitely many) prime ideals (Theorem \ref{productprime}), and subsequently complete the section by classifying the Leavitt path algebras in which every proper ideal can be factored in such a way (Theorem \ref{everyone-a-product-of-primes}). In light of Proposition \ref{quasi-primaryequivalent}, these results also cover factorizations of ideals into products of prime-power, primary, irreducible, and quasi-primary ones.

In Section \ref{semiprimesection}, which constitutes the heart of the article, we study the semiprime ideals of $L$. A proper ideal $I$ of a ring $R$ is called a \textit{semiprime} (or a \textit{radical}) ideal if whenever $J$ is an ideal of $R$ for which $J^2 \subseteq I$, then $J\subseteq I$. This is equivalent, among other conditions, to $I$ being the intersection of all the prime ideals of $R$ containing $I$. We show in Theorem \ref{Every I radical} that every proper ideal of $L$ is a semiprime ideal if and only if the underlying graph of $L$ satisfies Condition (K) (see definition in Section \ref{graphs-section}). Our main result of this section (Theorem \ref{Product of semiprimes}) gives a complete description, in terms generators, of the ideals of $L$ which can be factored as products of semiprime ideals. We also give necessary and sufficient conditions under which every proper ideal of $L$ can be factored as a product of semiprime ideals (Theorem \ref{Every I is prod. semiprimes}). In particular, $L$ has this property if $L$ is two-sided Noetherian and thus, for example, if the underlying graph is finite. 

Many of our results include graphical conditions and are illustrated by examples.  

\section{Preliminaries} \label{prelim-section}

We refer the reader to \cite{AAS} for the general notation, terminology, and results regarding Leavitt path algebras. Throughout, the word ``ideal" will  mean ``two-sided ideal" unless otherwise indicated. We begin  by recalling some of the basic concepts and results that will be needed. The reader familiar with the subject is encouraged to skip the rest of this section and refer back as necessary.

\subsection{Graphs} \label{graphs-section}

A \textit{directed graph} $E=(E^{0},E^{1},r,s)$ consists of two sets $E^{0}$ and $E^{1}$, together with functions $r,s:E^{1}\rightarrow E^{0}$, called \emph{range} and \emph{source}, respectively. The elements of $E^{0}$ are called \textit{vertices}, and the elements of $E^{1}$ are called \textit{edges}. We shall refer to directed graphs as simply ``graphs" from now on. Unless stated otherwise, we place no restrictions on the cardinalities of $E^{0}$ and $E^{1}$ in the graphs $E$ considered, except that $|E^{0}| > 0$.

A \emph{path} $\mu$ in a graph $E$ is a finite sequence $e_1\cdots e_n$ of edges $e_1,\dots, e_n \in E^1$ such that $r(e_i)=s(e_{i+1})$ for $i\in \{1,\dots,n-1\}$. Here we define $s(\mu):=s(e_1)$ to be the \emph{source} of $\mu$, $r(\mu):=r(e_n)$ to be the \emph{range} of $\mu$, and $|\mu|:=n$ to be the \emph{length} of $\mu$. We view the elements of $E^0$ as paths of length $0$ (extending $s$ and $r$ to $E^0$ via $s(v)=v = r(v)$ for all $v\in E^0$). The set of all vertices on a path $\mu$ is denoted by $\{\mu^{0}\}$. A path $\mu = e_1\cdots e_{n}$ is \textit{closed} if $r(e_{n})=s(e_{1})$, in which case $\mu$ is said to be \textit{based} at the vertex $s(e_{1})$. A closed path $\mu = e_1\cdots e_{n}$ is a \textit{cycle} if $s(e_{i})\neq s(e_{j})$ for all $i\neq j$. An \textit{exit} for a path $\mu = e_1\cdots e_{n}$ is an edge $f \in E^{1} \setminus \{e_1, \dots, e_n\}$ that satisfies $s(f)=s(e_{i})$ for some $i$. The graph $E$ is said to satisfy \textit{Condition (L)} if every cycle in $E$ has an exit. Also $E$ is said to satisfy \textit{Condition (K)} if any vertex on a closed path $\mu$ is also the base for a closed path $\gamma$ different from $\mu$ (i.e., one possessing a different set of edges). A cycle $\mu$ in $E$ is said to be \textit{without (K)} if no vertex along $\mu$ is the source of a different cycle in $E$.

Given a vertex $v \in E^{0}$, we say that $v$ is a \textit{sink} if $s^{-1}(v) = \emptyset$, that $v$ is \textit{regular} if $s^{-1}(v)$ is finite but nonempty, and that $v$ is an \textit{infinite emitter} if $s^{-1}(v)$ is infinite. A graph without infinite emitters is said to be
\textit{row-finite}. If $u,v \in E^0$ and there is a path $\mu$ in $E$ satisfying $s(\mu) = u$ and $r(\mu)=v$, then we write $u \geq v$. Given a vertex $v \in E^0$, we set $M(v) = \{w \in E^0 \mid w \geq v\}$. A nonempty subset $D$ of $E^{0}$ is said to be \textit{downward directed} if for any $u,v\in D$, there exists $w\in D$ such that $u\geq w$ and $v\geq w$. A subset $H$ of $E^0$ is \emph{hereditary} if whenever $u \in H$ and $u \geq v$ for some $v \in E^0$, then $v \in H$. Also $H \subseteq E^0$ is \emph{saturated} if $r(s^{-1}(v)) \subseteq H$ implies that $v \in H$ for any regular $v \in E^0$. A nonempty subset $M$ of $E^0$ is a \emph{maximal tail} if it satisfies the following three conditions.

(MT1) If $v\in M$ and $u\in E^0$ are such that $u \geq v$, then $u\in M$.

(MT2) For every regular $v \in M$ there exists $e \in E^1$ such that $s(e)=v$ and $r(e) \in M$. 

(MT3) $M$ is downward directed.\\
For any subset $H$ of $E^0$ it is easy to see that $H$ is hereditary if and only if $M=E^0\setminus H$ satisfies (MT1), and $H$ is saturated if and only if $M=E^0\setminus H$ satisfies (MT2).

\subsection{Leavitt Path Algebras}

Given a graph $E$ and a field $K$, the \textit{Leavitt path $K$-algebra} $L_K(E)$ \textit{of $E$} is the $K$-algebra generated by the set $\{v : v\in E^{0}\} \cup \{e,e^* : e\in E^{1}\}$, subject to the following relations:

\smallskip

{(V)} \ \ \ \  $vw = \delta_{v,w}v$ for all $v,w\in E^{0}$,

{(E1)} \ \ \ $s(e)e=er(e)=e$ for all $e\in E^{1}$,

{(E2)} \ \ \ $r(e)e^*=e^*s(e)=e^*$ for all $e\in E^{1}$,

{(CK1)} \ $e^*f=\delta _{e,f}r(e)$ for all $e,f\in E^{1}$, and

{(CK2)} \  $v=\sum_{e\in s^{-1}(v)} ee^*$ for all regular $v\in E^{0}$.   

\smallskip

\noindent Throughout this article, $K$ will denote an arbitrary field, $E$ will denote an arbitrary graph, and $L_K(E)$ will often be denoted simply by $L$. 

For all $v \in E^0$ we define $v^*:=v$, and for all paths $\mu  = e_1 \cdots e_n$ ($e_1, \dots, e_n \in E^1$) we set $\mu^*:=e_n^* \cdots e_1^*$, $r(\mu^*):=s(\mu)$, and $s(\mu^*):=r(\mu)$. With this notation, every element of $L_K(E)$ can be expressed (though not necessarily uniquely) in the form $\sum_{i=1}^n a_i\mu_i\nu_i^*$ for some $a_i \in K$ and paths $\mu_i,\nu_i$. We also note that while $L_K(E)$ is generally not unital, it has \textit{local units} for every choice of $E$. That is, for each finite subset $\{a_1, \dots, a_t\} \subseteq  L_K(E)$ there is an idempotent $u\in L_K(E)$ (which can be taken to be a sum of vertices in $E^{0}$) such that $ua_i=a_i=a_iu$ for all $1\leq i \leq t$.  

Every Leavitt path algebra $L_{K}(E)$ is $\mathbb{Z}$-graded (where $\mathbb{Z}$ denotes the group of  integers). Specifically, $L_{K} (E)=\bigoplus_{n\in\mathbb{Z}} L_{n}$, where 
\[
L_{n}=\bigg\{\sum_i k_{i}\mu_{i}\nu_{i}^*\in L_{K}(E) :\text{ }|\mu_{i}|-|\nu_{i}|=n\bigg\}.
\] 
Here the \emph{homogeneous components} $L_{n}$ are abelian subgroups satisfying $L_{m}L_{n}\subseteq L_{m+n}$ for all $m, n\in \mathbb{Z}$.  An ideal $I$ of $L_{K}(E)$ is said to be a \emph{graded ideal} if $I = \bigoplus_{n\in\mathbb{Z}} (I\cap L_{n})$. Equivalently, if $a\in I$ and $a=a_{i_{1}}+\cdots+a_{i_{m}}$ is a graded sum, with $a_{i_{k}}\in L_{i_{k}}$ for all $k \in \{1,\dots, m\}$ and $i_1, \dots, i_m$ distinct, then $a_{i_{k}}\in I$ for all $k$.

\subsection{Ideals in Leavitt Path Algebras} \label{LPAidealSect}

In this subsection we record various results from the literature and basic observations about ideals in Leavitt path algebras that will be used frequently in the paper.

Given a graph $E$, a \emph{breaking vertex} of a hereditary saturated subset $H$ of  $E^0$ is an infinite emitter $w\in E^{0}\backslash H$ with the property that $0<|s^{-1}(w)\cap r^{-1}(E^{0}\backslash H)|<\aleph_0$. The set of all breaking vertices of $H$ is denoted by $B_{H}$. For any $v\in B_{H}$, $v^{H}$ denotes the element $v-\sum_{s(e)=v, \, r(e) \notin H}ee^*$ of $L_K(E)$. Given a hereditary saturated subset $H$ of $E^0$ and $S\subseteq B_{H}$, $(H,S)$ is called an \emph{admissible pair}, and the ideal of $L_K(E)$ generated by $H\cup \{v^{H}:v\in S\}$ is denoted by $I(H,S)$. The graded ideals of $L_{K}(E)$ are precisely the ideals of the form $I(H,S)$ \cite[Theorem 2.5.8]{AAS}. Moreover, given two admissible pairs $(H_1,S_1)$ and $(H_2,S_2)$, setting $(H_1,S_1) \leq (H_2,S_2)$ whenever $H_1 \subseteq H_2$ and $S_1 \subseteq H_2 \cup S_2$, defines a partial order on the set of all admissible pairs of $L_K(E)$. The map $(H,S) \mapsto I(H,S)$ gives a one-to-one order-preserving correspondence between the partially ordered set of admissible pairs and the set of all graded ideals of $L_K(E)$, ordered by inclusion. Finally, for any admissible pair $(H,S)$ we have $L_{K}(E)/I(H,S)\cong L_{K} (E\backslash(H,S))$, via a graded isomorphism (i.e., one that takes each homogeneous component in one ring to the corresponding homogeneous component in the other ring) 
\cite[Theorem 2.4.15]{AAS}. Here $E\backslash(H,S)$ is a \emph{quotient graph of} $E$, where \[(E\backslash(H,S))^{0}=(E^{0}\backslash H)\cup\{v':v\in B_{H} \backslash S\}\]
and
\[(E\backslash(H,S))^{1}=\{e\in E^{1}:r(e)\notin H\}\cup\{e':e\in E^{1} \text{ with } r(e)\in B_{H}\backslash S\},\]
and $r,s$ are extended to $E\backslash(H,S)$ by setting $s(e^{\prime
})=s(e)$ and $r(e')=r(e)'$. We note that $(E\backslash(H,B_{H}))^{0} = E^{0} \backslash H$ for any  hereditary saturated $H$.

\begin{theorem} \label{LPAbasics}
Let $L=L_K(E)$ be a Leavitt path algebra.
\begin{enumerate}
\item[$(1)$] \cite[Proposition 2.9.9]{AAS} \ Every ideal of $L$ is graded if and only if $E$ satisfies Condition (K).  

\item[$(2)$] \cite[Proposition 2.3.2]{AAS} \ The Jacobson radical of $L$ is zero.
\end{enumerate}
\end{theorem}

Let $E$ be a graph, and $c$ a cycle in $E$ based at $v$.   For $f(x) = \sum_{i=0}^n k_i x^i \in K[x]$, we denote by $f(c)$ the element $k_0v + \sum_{i=1}^n k_i c^i \in L_K(E)$.  
   
Given a ring $R$ and a subset $S$ of $R$, we denote by $\langle S \rangle$ the ideal of $R$ generated by $S$.  

\begin{lemma}\label{Product} 
Let $L=L_K(E)$ be a Leavitt path algebra, let $c$ and $d$ be distinct cycles in $E$ (that is, $c$ and $d$ have distinct sets of edges) without exits, and let $f(x), g(x)\in K[x]$ be polynomials with nonzero constant terms.
\begin{enumerate}
\item[$(1)$] \cite[Lemma 3.3]{R2} \ $\langle f(c) \rangle \langle g(c) \rangle = \langle f(c)g(c) \rangle = \langle g(c)f(c)\rangle=\langle g(c)\rangle \langle f(c)\rangle$.  

\item[$(2)$] $\langle f(c)\rangle \langle g(d)\rangle =0$.
\end{enumerate}
\end{lemma}

\begin{proof}[Proof of (2).]
First, we note that since $c$ and $d$ are distinct cycles without exits, necessarily $\{c^{0}\}\cap\{d^{0}\}=\emptyset$. Now suppose, seeking a contradiction, that $a\in \langle f(c)\rangle \langle g(d)\rangle$ is nonzero. Then $a$ is a $K$-linear combination of nonzero terms of the form $\alpha\beta^{\ast}f(c)\gamma\delta^{\ast}g(d)\mu\nu^{\ast}$, where $\alpha,\beta,\gamma,\delta,\mu,\nu$ are paths in $E$. Since $a \neq 0$, for any such term we have $s(\gamma)\in\{c^{0}\}$ and $s(\delta)=r(\delta^{\ast})\in\{d^{0}\}$. Since $c$ and $d$ have no exits, this implies that $r(\gamma)\in\{c^{0}\}$ and $r(\delta)\in\{d^{0}\}$. But then $r(\gamma)=r(\delta)\in\{c^{0}\}\cap\{d^{0}\}$ contradicts $\{c^{0}\}\cap\{d^{0}\}=\emptyset$.
\end{proof}

\begin{theorem}\label{arbitrary ideals}
Let $L=L_K(E)$ be a Leavitt path algebra, and let $I$ be an ideal of $L$, with $H=I\cap E^{0}$ and $S=\{v\in B_{H}:v^{H}\in I\}$. 
\begin{enumerate}
\item[$(1)$] \cite[Theorem 4]{R1} \ $I=I(H,S)+\sum_{i\in Y} \langle f_{i}(c_{i})\rangle$ where $Y$ is a possibly empty index set; for each $i\in Y$, $c_{i}$ is a cycle without exits in $E\backslash(H,S)$; and $f_{i}(x)\in K[x]$ is a polynomial with a nonzero constant term, which is of smallest degree such that $f_{i}(c_{i})\in I$.

\item[$(2)$] \cite[Proposition 2.4.7]{AAS} \ Using the notation of (1), in $L/I(H,S)$ we have $I/I(H,S)=\bigoplus_{i\in Y} \langle f_{i}(c_{i})\rangle\subseteq \bigoplus_{i\in Y} \langle \{c_{i}^{0}\} \rangle$.

\item[$(3)$] \cite[Corollary 2.4.16]{AAS} \ $I(H,S) \cap E^{0} = H = \langle H \rangle \cap E^{0}$.
\end{enumerate}
\end{theorem}

If $I$ is an ideal as in Theorem \ref{arbitrary ideals}(1), then $I(H,S)$, also denoted $\mathrm{gr}(I)$, is called the \textit{graded part} of $I$.  

\begin{remark} \label{Uniqueness of cycles} 
We  point out that in the representation of a non-graded ideal $I=I(H,S)+ \sum_{i\in Y} \langle f_{i}(c_{i})\rangle$ given in Theorem \ref{arbitrary ideals}(1), the cycles $c_{i}$ are uniquely determined (up to a permutation of their vertices) by $I$. The uniqueness of the cycles can be derived from \cite[Proposition 6]{R1} and also \cite[Theorem 2.8.10]{AAS}. However, because of the importance of this fact  in our subsequent arguments, we shall outline the justification. 

Suppose that $I=I(H,S)+\sum_{j\in Y'} \langle g_{j}(d_{j})\rangle$ is another representation of $I$, where each $d_{j}$ is a cycle without exits in $E\backslash(H,S)$ and each $g_{j}(x)\in K[x]$ has a nonzero constant term. Then by Theorem \ref{arbitrary ideals}(2), we also have 
\[
I/I(H,S)= \bigoplus_{j\in Y'} \langle g_{j}(d_{j})\rangle\subseteq \bigoplus_{j\in Y'} \langle \{d_{j}^{0}\} \rangle.
\] 
Suppose that $d_{j}$ does not equal $c_{i}$ for any $i\in Y$. Then, by Lemma \ref{Product}(2),
\[
\langle g_{j}(d_{j})\rangle\cdot I/I(H,S)=\langle g_{j}(d_{j})\rangle\cdot \bigoplus_{i\in Y} \langle f_{i}(c_{i})\rangle=0,
\]
which yields that  $\langle g_{j}(d_{j})\rangle^{2}=0$ in the Leavitt path algebra $L_{K}
(E\backslash(H,S))$. This implies that the Jacobson radical of $L_{K}
(E\backslash(H,S))$ contains $\langle g_{j}(d_{j})\rangle$, and is hence nonzero, contrary to Theorem \ref{LPAbasics}(2).  Hence each $d_{j}$ is equal to some $c_{i}$ and conversely.  
\end{remark}

\begin{lemma} \label{Product=Intersection} \cite[Lemma 3.1]{R2}
Let $I$ be a graded ideal of a Leavitt path algebra $L$.
\begin{enumerate}
\item[$(1)$] $IJ=I\cap J=J\cap I=JI$ for any ideal $J$ of $L$. In particular, for any ideal $J \subseteq I$, we have $IJ=J$, and $I^{2}=I$.

\item[$(2)$] Let $I_{1},\dots, I_{n}$ be ideals of $L$, for some positive integer $n$. Then $I=I_{1}\cdots I_{n}$ if and only if $I=I_{1}\cap\cdots\cap I_{n}$.
\end{enumerate}
\end{lemma}

For a ring $R$ and a nonempty set $\Lambda$, the ring consisting of those $\Lambda \times \Lambda$ matrices over $R$ having at most finitely many nonzero entries is denoted by  ${M}_\Lambda(R)$.  

\begin{lemma} \cite[Lemma 2.7.1]{AAS} \label{cycle-ideal-lemma} 
Let $L=L_K(E)$ be a Leavitt path algebra, let $c$ be a cycle without exits in $E$, and let $M = \langle \{c^{0}\}\rangle$. Then $M \cong M_{\Lambda}(K[x,x^{-1}])$ for some index set $\Lambda$.
\end{lemma}

\begin{remark} \label{cycle-ideal-remark} 
We note, for future reference, that in the isomorphism $M \cong M_{\Lambda}(K[x,x^{-1}])$ constructed in the proof of \cite[Lemma 2.7.1]{AAS}, $c \in M$ is sent to a matrix in $M_{\Lambda}(K[x,x^{-1}])$ which has $x$ as one of the diagonal entries and zeros elsewhere.
\end{remark}

\begin{theorem} \label{specificprime}  \cite[Theorem 3.12]{R0}
Let $L=L_K(E)$ be a Leavitt path algebra, let $I$ be a proper ideal of $L$, and let $H = I \cap E^0$. Then $I$ is a prime ideal if and only if $I$ satisfies one of the following conditions.
\begin{enumerate}
\item[$(1)$] $I = I(H, B_{H})$, and $E^0\setminus H$ is downward directed. 
\item[$(2)$] $I = I(H, B_{H}\setminus \{u\})$ for some $u \in B_H$, and $E^0\setminus H = M(u)$.
\item[$(3)$] $I = I(H, B_{H}) + \langle f(c) \rangle$ where $c$ is a cycle without (K), $E^0\setminus H = M(s(c))$, and $f(x) \in K[x, x^{-1}]$ is an irreducible polynomial.
\end{enumerate}
\end{theorem}

\begin{theorem} \label{Radical ideals} \cite[Theorem 3.3]{AGMR} 
A proper ideal $I$ of a Leavitt path algebra $L$ is semiprime if and only if $I=I(H,S)+\sum_{i\in Y} \langle f_{i}(c_{i})\rangle$ as indicated in Theorem \ref{arbitrary ideals}(1), with the additional condition that each $f_{i}(x)\in K[x]$ is a square-free polynomial.
\end{theorem}

We conclude this section with a result about arbitrary rings.

\begin{proposition}\label{Morita equivalence} \cite[Proposition 1]{EKKRR}
Let $R$ be a ring with local units and $\Lambda$ a nonempty set.
\begin{enumerate}
\item[$(1)$] Every ideal of  $M_{\Lambda}(R)$ is of the form $M_{\Lambda}(I)$ for some ideal $I$ of $R$. The map $I \longmapsto M_{\Lambda}(I)$ defines a lattice isomorphism between the lattice of ideals of $R$ and the lattice of ideals of $M_{\Lambda}(R)$.

\item[$(2)$] For any two ideals $I$ and $J$ of $R$, we have $M_{\Lambda}(IJ)=M_{\Lambda} (I)M_{\Lambda}(J)$.
\end{enumerate}
\end{proposition}

\section{Multiplicative Conditions on Ideals}\label{MultCondSection} 

In this brief section we recall some multiplicative conditions on ideals, which will be useful for the remainder of the paper, and then discuss them in the context of Leavitt path algebras. These conditions have been studied extensively in the literature, especially in the commutative setting.

\begin{definition} Let $R$ be a ring with local units, and let $I$ be a proper ideal of $R$.   

\begin{enumerate}
\item[$(1)$] The \emph{radical} of  $I$, denoted $\mathrm{rad}(I)$ (or $\sqrt{I}$ elsewhere in the literature), is the intersection of all the prime ideals of $R$ containing $I$.  (It is not hard to show that every proper ideal in a ring with local units is contained in a maximal ideal, and hence in a prime ideal, so that this set of ideals is necessarily nonempty.)  

\item[$(2)$] $I$ is \emph{primary} in case for all ideals $A$ and $B$ of $R$, $B\subseteq \mathrm{rad}(I)$ whenever $AB \subseteq I$ and $A\not\subseteq I$ .  

\item[$(3)$] $I$ is \emph{quasi-primary} in case $\mathrm{rad}(I)$ is prime.

\item[$(4)$] $I$ is \emph{irreducible} in case for all ideals $A$ and $B$ of $R$, $I = A\cap B$ implies that $I = A$ or $I=B$.   

\item[$(5)$] $I$ is a \emph{prime power} in case $I = P^n$ for some prime ideal $P$ of $R$ and some positive integer $n$.   
\end{enumerate}
\end{definition}

Note that a prime ideal $P$ in any ring is irreducible, since $I\cap J=P$ implies that $IJ\subseteq P$, which gives  $I\subseteq P$ or $J\subseteq P$, and hence $I=P$ or $J=P$. It is easy to see that any prime ideal also satisfies the conditions ``primary", ``quasi-primary", and ``prime power".

While conditions (2)--(5) in the above definition are generally not interchangeable, they happen to coincide in Leavitt path algebras.

\begin{proposition}\label{quasi-primaryequivalent}   
Let $L=L_K(E)$ be a Leavitt path algebra, let $I$ be a proper ideal of $L$, and write $\mathrm{gr}(I) = I(H,S)$. Then the following are equivalent.
\begin{enumerate}
\item[$(1)$] $I$ is primary.

\item[$(2)$] $I$ is quasi-primary.

\item[$(3)$] $I$ is irreducible.

\item[$(4)$] $I$ is a prime power.  

\item[$(5)$] Either    
\begin{enumerate}
\item[$(5.1)$] $I$ is a graded prime ideal, in which case $(E\setminus (H,S))^{0}$ is downward directed, or   
\item[$(5.2)$] $I$ is a power of a non-graded prime ideal, in which case $(E\setminus (H,S))^{0}$ is downward directed; and $I = I(H,B_{H}) + \langle p^n(c) \rangle$ for a $($unique$)$ cycle $c$ without exits in $E\setminus (H,B_{H})$, an irreducible polynomial $p(x) \in K[x,x^{-1}]$, and a positive integer $n$.
\end{enumerate}
\end{enumerate}
\end{proposition}

\begin{proof}
The equivalences (1) $\Leftrightarrow$ (3) $\Leftrightarrow$ (4)  are established in      \cite[Theorem 5.7]{R2}. We shall show that (2) $\Rightarrow$ (5) $\Rightarrow$ (4) $\Rightarrow$ (2).

(2) $\Rightarrow$ (5) Suppose that $I$ is quasi-primary, so that $\mathrm{rad}(I)=P$ is a prime ideal. Suppose further that $I$ is graded. Then $I = I(H,S)$ is semiprime, by Theorem \ref{Radical ideals}, and hence $I=\mathrm{rad}(I)$ is itself prime. As mentioned in Section \ref{LPAidealSect}, $L_K(E)/I \cong L_K(E\setminus (H,S))$. Since the zero ideal in $L_K(E\setminus (H,S))$ is prime, it follows from Theorem \ref{specificprime} that $(E\setminus (H,S))^{0}$ is downward directed, establishing  (5.1) in case $I$ is graded.

Let us therefore assume that $I$ is not graded, and, using Theorem \ref{arbitrary ideals}(1), write $I=I(H,S)+ \sum_{i\in Y} \langle f_{i}(c_{i})\rangle$, where each $c_{i}$ is a cycle
without exits in $E\backslash(H,S)$, and each $f_{i}(x)\in K[x]$ has a nonzero constant term. According to \cite[Lemma 5.4]{R2}, $\mathrm{gr}(I)=\mathrm{gr}(\mathrm{rad}(I))=\mathrm{gr}(P)$, and so $\mathrm{gr}(I)=I(H,S)$ is a prime ideal, by Theorem \ref{specificprime}. Since necessarily $\mathrm{gr}(P) \neq P$, by the same theorem, 
\[
P = I(H,S)+\langle p(c)\rangle = I(H,B_{H})+\langle p(c)\rangle,
\] 
where $c$ is a cycle without (K) in $E$, $E^0\setminus H = M(s(u))$, and $p(x)\in K[x,x^{-1}]$ is irreducible (which can be taken to be in an irreducible polynomial in $K[x]$). In particular, $E^{0}\setminus H = (E\backslash(H,B_{H}))^{0}$ is downward directed, and hence there can be only one cycle without exits in $E\backslash(H,B_{H})$. So $I = I(H,B_{H})+\langle f(d)\rangle$ for such a cycle $d$ and some $f(x)\in K[x]$ with a nonzero constant term. Since $E^0\setminus H = M(s(c))$, $s(c)$ must lie on $d$, and since $c$ is a cycle without (K), it follows that $c=d$.

We claim that $f(x)=ap^{n} (x)$ for some positive integer $n$ and $a \in K$. Suppose, on the contrary, that there is an irreducible polynomial that is not conjugate to $p(x)$ (i.e., $q(x) \in K[x] \setminus Kp(x)$), necessarily with a nonzero constant term, which is a divisor of $f(x)$. Then, again by Theorem \ref{specificprime}, $Q=I(H,B_H)+\langle q(c)\rangle$ is a prime ideal containing $I$. But then $Q\supseteq \mathrm{rad}(I)=P$, which implies that $p(c) \in Q$. Since $p(x)$ and $q(x)$ are not $K$-scalar multiples of each other, and since $K[x]$ is a Euclidean domain, $1 = p(x)a(x)+q(x)b(x)$ for some $a(x), b(x) \in K[x]$. Since $p(c)\in Q$ and $q(c) \in Q$, this implies that $c^0=s(c)$ is in $Q \cap E^{0} = H$, by Theorem \ref{arbitrary ideals}(3), contrary to the choice of $c$. Thus $p(x)$ is the only possible irreducible divisor of $f(x)$, up to multiplying by a constant, and so $f(x)=ap^{n}(x)$ for some $n$ and $a \in K$. It follows that $I=I(H,B_H)+\langle p^{n}(c)\rangle$. Finally, by Theorem \ref{arbitrary ideals}(2) and Lemma \ref{Product}(1), we have $P^{n} = I(H,B_H)+\langle p^{n}(c)\rangle$, and so $I = P^{n}$.

(5) $\Rightarrow$ (4) This is a tautology.

(4) $\Rightarrow$ (2) If $I=P^{n}$, for some prime ideal $P$ and positive integer $n$, then every prime
ideal $Q$ containing $I$ contains $P$. So ${\rm rad}(I)=P$, and thus  $I$
is quasi-primary.
\end{proof}

While any prime ideal satisfies the properties given in Proposition \ref{quasi-primaryequivalent}, the next two examples show that these properties are independent of the ideal $I$ being semiprime.     

\begin{example}\label{primenotprimaryremark}
Let $E$ be the graph having one vertex $v$ and one loop $e$, with $s(e)=v=r(e)$, pictured below. 
\smallskip
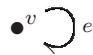
\[
\xymatrix{{\bullet}^{v} \ar@(ur,dr)^e}
\]
\smallskip

\noindent Then $K[x,x^{-1}]\cong L_{K}(E)$, via the map induced by sending $1\mapsto v$, $x\mapsto e$, and $x^{-1}\mapsto e^*$. In this ring the ideal $I = \langle (x-1)^2 \rangle$ is a prime power, but clearly not semiprime. \hfill $\Box$
\end{example}

\begin{example}
Letting $E$ be a graph having two vertices and no edges, we have $K\oplus K \cong L_{K}(E)$. The zero ideal in this ring is easily seen to be semiprime but not a prime power. \hfill $\Box$
\end{example}


\section{Prime Ideals}\label{primesection}

In this section we examine prime ideals and products of prime ideals in Leavitt path algebras, as well as products of ideals of the sort described in Proposition \ref{quasi-primaryequivalent}. We remind the reader that a description of the prime ideals in these rings can be found in Theorem \ref{specificprime}.  As is standard, an ideal being a   ``product of prime ideals" includes the possibility that the ideal itself is prime.  

We begin with a characterization of those Leavitt path algebras for which every proper ideal is prime, which builds on an earlier description \cite[Proposition 3.6]{AGMR} of Leavitt path algebras where every semiprime ideal is prime and extends \cite[Proposition 2.7]{EKR}

\begin{theorem} \label{Everyone prime} 
The following are equivalent for any Leavitt path algebra $L=L_K(E)$.
\begin{enumerate}
\item[$(1)$] Every proper ideal of $L$ is prime.

\item[$(2)$] Every proper ideal of $L$ is a prime power.

\item[$(3)$] Every ideal of $L$ is graded, and the ideals of $L$ form a chain under set inclusion.

\item[$(4)$] The graph $E$ satisfies Condition (K), and the admissible pairs $(H,S)$ form a chain under the partial order of the admissible pairs.
\end{enumerate}
\end{theorem}

\begin{proof}
Suppose that (1) holds.  Then, in particular, every semiprime ideal is prime, and hence, by the equivalence of (1) and (4) in \cite[Proposition 3.6]{AGMR}, (4) holds.  Conversely, if (4) holds, then, by the equivalence of (3) and (4) in \cite[Proposition 3.6]{AGMR}, every (proper) ideal is prime, proving (1). So it suffices to show that (1)--(3) are equivalent.

(1) $\Rightarrow$ (2) This is a tautology.

(2) $\Rightarrow$ (3) Seeking a contradiction, suppose that (2) holds and there is a non-graded ideal $I$ in $L$. By Theorem \ref{arbitrary ideals}(1), $I=I(H,S)+ \sum_{i\in Y} \langle f_{i}(c_{i})\rangle$, where each $c_{i}$ is a cycle without exits in $E\backslash(H,S)$, and each $f_{i}(x)\in K[x]$ has a nonzero constant term. Fix $i\in Y$, and let $p(x), q(x) \in K[x,x^{-1}]$ be non-conjugate  irreducible polynomials. Then, using Lemma \ref{Product}(1) and Theorem \ref{arbitrary ideals}(2),
\[
I(H,S)+\langle p(c_{i})q(c_{i})\rangle=(I(H,S)+\langle p(c_{i})\rangle)(I(H,S)+\langle q(c_{i})\rangle).
\]
This is a proper ideal, which, by Proposition, \ref{quasi-primaryequivalent} is not a prime power, contrary to (2). Hence every ideal of $L$ is graded. Moreover, by the same proposition, it must be the case that every ideal of $L$ is prime.

Now suppose that there are ideals $I$ and $J$ in $L$ such that $I\cap J\neq I$ and $I \cap J \neq J$. Then $I \cap J$ is a proper ideal, $IJ \subseteq I\cap J$, but $I \nsubseteq I\cap J$ and $J\nsubseteq I\cap J$, which contradicts $I \cap J$ being prime. Thus for all ideals $I$ and $J$ of $L$, either $I \cap J = I$ or $I \cap J = J$. It follows that the ideals of $L$ form a chain under set inclusion.

(3) $\Rightarrow$ (1) Suppose that (3) holds, and let $A, B, I$ be ideals of $L$ such that $AB \subseteq I$. Since the ideals form a chain, we may assume, without loss of generality, that $A \subseteq B$. Since $B$ is graded, $AB = A$, by Lemma \ref{Product=Intersection}(1), and hence $A \subseteq I$. Thus, every proper ideal of $L$ is prime.
\end{proof}

Of course, by Proposition \ref{quasi-primaryequivalent}, we could replace 	``prime power" in condition (2) above with any of ``primary", ``quasi-primary", or ``irreducible", while preserving equivalence.

There are obvious examples of Leavitt path algebras which satisfy the conditions of Theorem \ref{Everyone prime}, e.g. $K \cong L_K(\bullet)$.   We offer now an example in which the graph-theoretic condition (4) in Theorem \ref{Everyone prime} becomes the germane one to analyze.  

\begin{example} \label{chain of ideals}
Let $E$ be the following graph.

\smallskip
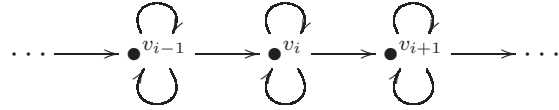
\[ 
\xymatrix{ \cdots  \ar[r] &  \bullet^{v_{i-1}} \ar@(ul,ur) \ar@(dr,dl) \ar[r]  &  \bullet^{v_{i}} \ar@(ul,ur) \ar@(dr,dl) \ar[r]  &  \bullet^{v_{i+1}} \ar@(ul,ur) \ar@(dr,dl) \ar[r] & \cdots }
\]
\smallskip
 
\noindent
Then clearly $E$ satisfies Condition (K), and is row-finite. In particular, $B_H$ is empty for every hereditary saturated $H \subseteq E^0$. Moreover, the nonempty proper hereditary saturated subsets of $E^0$ are precisely those of the form $\{v_{i} : i \geq n\}$, where $n \in \mathbb{Z}$, and hence form a chain under set inclusion. It follows that the admissible pairs $(H,S)$ form a chain under the partial order of the admissible pairs, described in Section \ref{LPAidealSect}, and so $E$ satisfies condition (4) in Theorem \ref{Everyone prime}. Thus $L_K(E)$ satisfies conditions (1)--(3) in the theorem. \hfill $\Box$
\end{example}

Our goal for the remainder of this section is to establish results analogous to Theorems \ref{specificprime} and \ref{Everyone prime}, with ``prime" replaced by ``product of primes", where by a ``product" of ideals we shall always mean a finite product.  To give  an analogue of Theorem \ref{specificprime} we require a lemma.  

Recall that given a collection $\{S_i : i \in X\}$ of sets, the intersection $\bigcap_{i\in X} S_i$ is \emph{irredundant} if $\bigcap_{i\in X \setminus \{j\}} S_i \not\subseteq S_j$ for all $j \in X$. Similarly, the union $\bigcup_{i\in X} S_i$ is \emph{irredundant} if $S_j \not\subseteq \bigcup_{i\in X \setminus \{j\}} S_i$ for all $j \in X$.   When the indexing set is finite, it is easy to show that any intersection (respectively, union) may be replaced by an indexing subset for which the intersection (respectively, union) is irredundant.  

\begin{lemma}\label{gradedprodlemma}
The following are equivalent for any Leavitt path algebra $L=L_K(E)$ and positive integer $n$.
\begin{enumerate}
\item[$(1)$] The zero ideal is the (irredundant) intersection of $n$ prime ideals.

\item[$(2)$] The zero ideal is the (irredundant) intersection of $n$ graded prime ideals.

\item[$(3)$] $E^{0}$ is the (irredundant) union of $n$ maximal tails.
\end{enumerate}
Moreover, the maximal tails in (3) can be taken to be the complements in $E^0$ of the sets of vertices contained in the the prime ideals in (1) or (2).
\end{lemma}

\begin{proof}
(1) $\Rightarrow$ (2) Suppose that $\{0\}$ is the intersection of $n$ prime ideals. Then, being graded, $\{0\}$ is the intersection of the graded parts of those prime ideals. Since the graded part of any prime ideal is prime, by Theorem \ref{specificprime}, $\{0\}$ is the intersection of $n$ graded prime ideals.

(2) $\Rightarrow$ (1) This is a tautology. 

(3) $\Rightarrow$ (2) Suppose that $E^0 = \bigcup_{i=1}^n M_i$ for some maximal tails $M_i \subseteq E^0$. Writing $H_i = E^0 \setminus M_i$, by Theorem \ref{specificprime}, we see that $P_i = I(H_i, B_{H_i})$ is a (graded) prime ideal of $L$, for each $i$. Now, $\emptyset = E^0 \setminus \bigcup_{i=1}^n M_i = \bigcap_{i=1}^n H_i$, and so, using Theorem \ref{arbitrary ideals}(3),
\[
\emptyset = \bigcap_{i=1}^n H_i = \bigcap_{i=1}^n (E^0 \cap P_i) = E^0 \cap \bigcap_{i=1}^n P_i.
\]
Since the intersection of a collection of graded ideals is graded, and since any nonzero graded ideal contains a vertex (since it must be generated by an admissible pair, as mentioned in Section \ref{LPAidealSect}), it follows that $\{0\} = \bigcap_{i=1}^n P_i$. Clearly, if the union $\bigcup_{i=1}^n M_i$ is irredundant, then so is the intersection $\bigcap_{i=1}^n H_i$, and hence so is $\bigcap_{i=1}^n P_i$.

(2) $\Rightarrow$ (3) Suppose that $\{0\} = \bigcap_{i=1}^n P_i$ for some graded prime ideals $P_i = I(H_i, S_i)$ of $L$. Then, using Theorem \ref{arbitrary ideals}(3),
$\emptyset = E^0 \cap \bigcap_{i=1}^n P_i = \bigcap_{i=1}^n H_i$. Hence $E^0 = E^0 \setminus \bigcap_{i=1}^n H_i = \bigcup_{i=1}^n (E^0 \setminus H_i)$, where each $E^0 \setminus H_i$ is a maximal tail, by Theorem \ref{specificprime}. We note that if the intersection $\bigcap_{i=1}^n P_i$ is irredundant, then so is $\bigcap_{i=1}^n H_i$ (since otherwise the intersection of $n-1$ of the $P_i$ would contain no vertices, and would therefore be zero), and hence so is the union $\bigcup_{i=1}^n (E^0 \setminus H_i)$.

The final claim follows from the constructions above.
\end{proof}

The next result extends \cite[Theorem 6.2]{R2}.
 
\begin{theorem}\label{productprime}  
Let $L=L_K(E)$ be a Leavitt path algebra, and let $I$ be a proper ideal of $L$. Then the following are equivalent.
\begin{enumerate}
\item[$(1)$] $I$ is a product of prime ideals.

\item[$(2)$] $I = I(H,S) + \sum_{i=1}^k \langle f_i(c_i)\rangle$, where each $c_i$ is a cycle without exits in $E \setminus (H,S)$; each $f_i(x) \in K[x]$ has a nonzero constant term; $(E \setminus (H,S))^{0}$ is the irredundant union of $n>0$ maximal tails; and $0 \leq k\leq n$, with $k=0$ indicating that $I=I(H,S)$.
\end{enumerate}
\end{theorem}

\begin{proof}
First, suppose that $I$ is graded. Then, by Lemma \ref{Product=Intersection}(2), $I$ is a product of prime ideals if and only if $I$ is the intersection of those same prime ideals. It follows that $I$ is a product of prime ideals of $L$ if and only if the zero ideal of $\bar{L} = L_K(E)/I$ is a finite intersection of prime ideals of $\bar{L}$. Now, as mentioned in Section \ref{LPAidealSect}, $\bar{L} = L_K(E)/I(H,S) \cong L_K(E\setminus (H,S))$.  Thus, by Lemma \ref{gradedprodlemma}, $I$ is a product of prime ideals if and only if $(E \setminus (H,S))^{0}$ is the union of finitely many maximal tails, which as mentioned previously may be assumed to be  irredundant. Thus (1) and (2) are equivalent in the case where $I$ is graded.

Now suppose that $I$ is not graded. Then, according to \cite[Theorem 6.2]{R2}, $I$ is a product of prime ideals if and only if $I(H,S)$ is the irredundant intersection of $n > 0$ prime ideals; and  $I/I(H,S) = \bigoplus_{i=1}^k \langle f_i(c_i)\rangle$, where $k\leq n$, each $c_i$ is a cycle without exits in $E \setminus (H,S)$, and $f_i(x) \in K[x]$ is a polynomial of smallest degree such that $f_i(c_i) \in I$ (which necessarily has a nonzero constant term). Now, by Theorem \ref{arbitrary ideals}(2), in this situation $I = I(H,S) + \sum_{i=1}^k \langle f_i(c_i)\rangle$ if and only if $I/I(H,S) = \bigoplus_{i=1}^k \langle f_i(c_i)\rangle$. Also, $I(H,S)$ being the irredundant intersection of $n$ prime ideals is equivalent to the zero ideal of $L/I(H,S) \cong L_K(E\setminus (H,S))$ being the irredundant intersection of $n$ prime ideals. By Lemma \ref{gradedprodlemma}, this is equivalent to $(E \setminus (H,S))^{0}$ being the irredundant union of $n$ maximal tails. It follows that (1) and (2) are equivalent when $I$ is not graded.
\end{proof}

With Theorem \ref{productprime} in hand, for contrast we give an example of an ideal in a Leavitt path algebra which is not a product of prime ideals.  

\begin{example}\label{exampleidealnotproductprimes}
Let $E$ be the graph having countably many vertices (indexed by the set $\mathbb{Z}^+$ of positive integers) and no edges. Then $L_K(E) \cong \bigoplus_{i\in \mathbb{Z}^+} K$, the infinite ring direct sum of copies of $K$. For each $j\in \mathbb{Z}^+$ let $P_j$ denote the ideal of  $\bigoplus_{i\in \mathbb{Z}^+} K$ consisting of those elements which are $0$ in the $j^{th}$ coordinate. It is straightforward to show that $\{P_j : j\in \mathbb{Z}^+ \}$ is precisely the set of prime ideals of  $\bigoplus_{i\in \mathbb{Z}^+} K$. It follows that the ideal $\{0\}$ of $\bigoplus_{i\in \mathbb{Z}^+} K$ is not the product of prime ideals.  \hfill $\Box$
\end{example}

Next, we give a characterization of the Leavitt path algebras where every proper ideal is a product of prime ideals. Rings with this property are sometimes referred to as \textit{general ZPI rings}, particularly in the literature on commutative rings.

\begin{theorem} \label{everyone-a-product-of-primes}
The following are equivalent for any Leavitt path algebra $L=L_K(E)$.

\begin{enumerate}
\item[$(1)$] Every proper ideal of $L$ is a product of prime ideals.

\item[$(2)$] There are only finitely many prime ideals minimal over any proper non-prime ideal of $L$.

\item[$(3)$] Every proper homomorphic image of $L$ is either a prime ring or possesses only a finite number of minimal prime ideals.

\item[$(4)$] For every admissible pair $(H,S)$ with $H \neq E^{0}$, $(E \setminus (H,S))^{0}$ is the irredundant union of $n>0$ maximal tails, and there are at most $n$ cycles without exits in $E \setminus (H,S)$. 
\end{enumerate}
\end{theorem}

\begin{proof}
(1) $\Leftrightarrow$ (3) is proved in \cite[Theorem 6.5]{R2}, and (2) $\Leftrightarrow$ (3) is immediate. We shall show that (1) $\Leftrightarrow$ (4).

(4) $\Rightarrow$ (1) Suppose that (4) holds, and let $I$ be a proper ideal of $L$. By Theorem \ref{arbitrary ideals}(1), we can write $I=I(H,S)+ \sum_{i\in Y} \langle f_{i}(c_{i})\rangle$, where each $c_{i}$ is a cycle without exits in $E \backslash (H,S)$, and each $f_{i}(x)\in K[x]$ has a nonzero constant term. By hypothesis, $|Y|$ is finite, and so the desired conclusion follows from Theorem \ref{productprime}.

(1) $\Rightarrow$ (4) We shall prove the contrapositive. First, suppose that there is an admissible pair $(H,S)$ with $H \neq E^{0}$, such that $(E \setminus (H,S))^{0}$ is not the union of finitely many maximal tails. Then, by Theorem \ref{productprime}, $I(H,S)$ is not a product of prime ideals.
 
Next, suppose that there is an admissible pair $(H,S)$ with $H \neq E^{0}$, such that $(E \setminus (H,S))^{0}$ is the irredundant union of $n>0$ maximal tails, but there are more than $n$ cycles without exits in $E \setminus (H,S)$. Let $\{c_i : 1\leq i \leq n+1\}$ be a collection of distinct cycles without exits in $E \setminus (H,S)$. Then, by Theorem \ref{arbitrary ideals}(1), $I = I(H,S) + \sum_{i =1}^{n+1} \langle s(c_{i}) + c_{i}\rangle$ is an ideal in $L$, which, by Theorem \ref{productprime}, is not a product of prime ideals.

Thus if (4) does not hold, then neither does (1).
\end{proof}

Clearly, an ideal in a ring is a product of prime ideals if and only if it is a product of prime power ideals. So the results analogous to Theorems \ref{productprime} and \ref{everyone-a-product-of-primes}, with ``prime" in statement (1) replaced by any of the properties appearing in Proposition \ref{quasi-primaryequivalent}, are in fact identical to those two theorems.

\begin{example}
Let $E$ be a graph with one vertex and one loop. Then, the only proper hereditary saturated subset of $E^{0}$ is the empty set. It follows that $E$ satisfies condition (4) in Theorem \ref{everyone-a-product-of-primes}, and hence every proper ideal in $L_K(E)$ is a product of prime ideals. But, as mentioned in Example \ref{primenotprimaryremark}, not every ideal in this ring is prime. 

Alternatively, identifying  $L_K(E)$ with $K[x,x^{-1}]$, as in Example \ref{primenotprimaryremark}, it is easy to deduce that every ideal in this ring is a product of prime ideals, from the fact that it is a principal ideal domain. \hfill $\Box$
\end{example}

We conclude this section with a couple of observations about the relationship between products and intersections of prime ideals, which will be useful for subsequent arguments.

\begin{proposition} \label{intersection-implies-product}
Let $L$ be a Leavitt path algebra, and let $I$ be a proper ideal of $L$. If $I$ is the intersection of finitely many prime ideals, then $I$ is a product of prime ideals.
\end{proposition}

\begin{proof}
Suppose that $I= \bigcap_{i=1}^{n} P_{i}$ for some prime ideals $P_{1}, \dots, P_{n}$. Then, using the fact that, by Theorem \ref{specificprime}, each $\mathrm{gr}(P_{i})$ is prime, together with Lemma \ref{Product=Intersection}(2), we have 
\[
\mathrm{gr}(I)= \bigcap_{i=1}^{n} \mathrm{gr}(P_{i}) =\mathrm{gr}(P_{1})\cdots  \mathrm{gr}(P_{n}).
\]
By \cite[Theorem 6.2]{R2}, $\mathrm{gr}(I)$ being a product of (graded) prime ideals implies that $I$ is a product of prime ideals.
\end{proof}

In contrast to Proposition \ref{intersection-implies-product}, in general, a product of (finitely many) prime ideals in a Leavitt path algebra need not be an intersection of (finitely or infinitely many) prime ideals, as the next lemma shows.  

\begin{lemma} \label{ProductNOTintersection}
Let $E$ be a graph having one vertex and one loop, let $L=L_K(E)$, and let $P$ be a nonzero prime ideal of $L$. Then $P^{2}$ is not the intersection of any collection of prime ideals of $L$.

Consequently, identifying $L$ with $K[x,x^{-1}]$, the analogous result holds in $M_\Lambda(K[x,x^{-1}])$, for any nonempty index set $\Lambda$. 
\end{lemma}

\begin{proof}
Recall that, as noted in Example \ref{primenotprimaryremark}, $L \cong K[x,x^{-1}]$. Now let $P$ be a nonzero prime ideal of $K[x,x^{-1}]$. It is well-known and easy to see (since $K[x,x^{-1}]$ is a principal ideal domain) that $P$ is necessarily maximal and $P^{2} \neq P$.  Suppose that $P^{2}= \bigcap_{i\in Y} Q_{i}$, for some prime ideals $Q_{i}$ of $K[x,x^{-1}]$.  Then, for each $i$, $Q_{i}\supseteq P^{2}$ implies that $Q_{i}\supseteq P$, and hence $Q_{i}=P$, as $P$ is maximal. But then $P^{2}= \bigcap_{i\in Y} P=P$ contradicts $P^{2} \neq P$, and hence $P^{2}$ is not the intersection of any collection of prime ideals.  

The second statement follows from Proposition \ref{Morita equivalence}.  
\end{proof} 

With Lemma \ref{ProductNOTintersection} in mind, we note that Theorem \ref{Every I radical} below gives criteria under which every product of prime ideals in a Leavitt path algebra is the intersection of (those same) prime ideals.  

%
%

\section{Semiprime Ideals}\label{semiprimesection}

Having analyzed the prime and ``product of prime" ideals in the previous section, we now turn our attention to a similar analysis in the context of semiprime and ``product of semiprime" ideals.  As is standard, an ideal being a ``product of semiprime ideals" includes the possibility that the ideal itself is semiprime. We shall give a description of the Leavitt path algebras where every proper ideal is semiprime (Theorem \ref{Every I radical}), a description of the ideals in an arbitrary Leavitt path algebra which can be written as products of semiprime ideals (Theorem \ref{Product of semiprimes}), and a classification of those graphs $E$ for which every proper ideal of $L_K(E)$ admits such a factorization (Theorem \ref{Every I is prod. semiprimes}).  

We begin with a technical lemma and some notation.

\begin{lemma} \label{intersect-lemma}
Let $L=L_K(E)$ be a Leavitt path algebra, and let $I \subseteq J$ be ideals of $L$, such that $I$ is not graded. Write $I=I(H,S)+\sum_{i\in Y} \langle f_{i}(c_{i})\rangle$, using the notation of Theorem \ref{arbitrary ideals}(1). Also, let $M = \bigoplus_{i \in Y} \langle \{c_i^0\}\rangle$ in $L/I(H,S)\cong L_{K} (E\backslash(H,S))$. Then exactly one of the following holds.
\begin{enumerate}
\item[$(1)$] $M\cap \mathrm{gr}(J/I(H,S)) = M$.

\item[$(2)$] $M\cap \mathrm{gr}(J/I(H,S)) = \{0\}$, and $M\cap J/I(H,S) = \sum_{i\in Y'}\langle g_{i}(c_{i})\rangle$ for some $Y' \subseteq Y$ and $g_{i}(x)\in K[x]$, having nonzero constant terms.
\end{enumerate}
\end{lemma}

\begin{proof}
First, we note that in $L/I(H,S)\cong L_{K} (E\backslash(H,S))$, by Theorem \ref{arbitrary ideals}(2) and Lemma \ref{cycle-ideal-lemma}, we have
\[
\bar{I}=I/I(H,S)=\bigoplus_{i\in Y} \langle f_{i}(c_{i})\rangle\subseteq M= \bigoplus_{i\in Y}\langle \{c_i^0\}\rangle \cong \bigoplus_{i\in Y} M_{\Lambda_{i}}(K[x,x^{-1}])
\]
for some index sets $\Lambda_{i}$. 

Now, using Theorem \ref{arbitrary ideals}(1), we can write $J = I(H',S')+\sum_{j\in X} \langle g_{j}(d_{j})\rangle$, where each $d_{j}$ is a cycle without exits in $E\backslash(H',S')$, and each $g_{j}(x)\in K[x]$ has a nonzero constant term. Letting $\bar{J} = J/I(H,S)$, we note that the ideal $M\cap \mathrm{gr}(\bar{J})$ of $L/I(H,S) \cong L_{K} (E\backslash(H,S))$ is graded, and hence, by Lemma \ref{Product=Intersection}(1), $(M\cap \mathrm{gr}(\bar{J}))^{2}=(M\cap \mathrm{gr}(\bar{J}))$. It is easy to see that $A^{2}\neq A$ for any nonzero proper ideal $A$ of $K[x,x^{-1}]$, and so, by Proposition \ref{Morita equivalence}, this property holds for ideals in any matrix ring $M_{\Lambda}(K[x,x^{-1}])$, and hence also in $M\cong \bigoplus_{i\in Y} M_{\Lambda_{i}}(K[x,x^{-1}])$. It follows that either $M\cap \mathrm{gr}(\bar{J}) = \{0\}$ or $M\cap \mathrm{gr}(\bar{J}) = M$.

To conclude the proof, let us assume that $M\cap \mathrm{gr}(\bar{J}) = \{0\}$ and describe $M\cap \bar{J}$. Using the distributive law for ideals of a Leavitt path algebra \cite[Theorem 4.3]{R2}, we have
\begin{align*}
M\cap\bar{J} & = [M\cap \mathrm{gr}(\bar{J})]+\sum_{j\in X} [M\cap\langle g_{j}(d_{j})\rangle] \\
& = \{0\} + \sum_{j\in X}[M\cap\langle g_{j}(d_{j})\rangle] =\sum_{j\in X}\sum_{i\in Y} [ \langle \{c_i^0\}\rangle \cap\langle g_{j}(d_{j})\rangle].
\end{align*}
We claim that for all $i$ and $j$, either $\langle \{c_i^0\}\rangle \cap\langle g_{j}(d_{j})\rangle = \{0\}$ or $c_i = d_j$, in which case $\langle \{c_i^0\}\rangle \cap\langle g_{j}(d_{j})\rangle=\langle g_{j}(d_{j})\rangle$. Thus, suppose that $G=\langle \{c_i^0\}\rangle \cap\langle g_{j}(d_{j})\rangle\neq \{0\}$. Since $M\cap \mathrm{gr}(\bar{J}) = \{0\}$, it follows that $s(c_{i})\in E^{0}\backslash H'$, and hence both $c_{i}$ and $d_{j}$ are cycles without exits in $E\backslash (H',S')$. Next, recall that $\langle \{c_{i}^{0}\}\rangle \cong M_{\Lambda_{i}}(K[x,x^{-1}])$, and, as noted in Remark \ref{cycle-ideal-remark}, the isomorphism involved sends $c_{i}$ to a matrix in $M_{\Lambda_{i}}(K[x,x^{-1}])$ which has $x$ as one of the entries and zeros elsewhere. Thus, it follows from Proposition \ref{Morita equivalence} and $K[x,x^{-1}]$ being a principal ideal domain that $G=\langle h(c_{i})\rangle$, for some nonzero $h(x)\in K[x,x^{-1}]$. Now, suppose that $c_{i}\neq d_{j}$. Then, in the quotient ring $L/I(H',S') \cong L_{K}(E\backslash (H',S'))$ of $L/I(H,S)$, by Lemma \ref{Product}(2), we have $\langle h(c_{i}) \rangle \langle g_{j}(d_{j}) \rangle = \{0\}$, since $c_{i}$ and $d_{j}$ are cycles without exits in $E\backslash (H',S')$. As $\langle h(c_{i}) \rangle \subseteq \langle g_{j}(d_{j}) \rangle$, this implies that $\langle h(c_{i}) \rangle \neq 0$ but $\langle h(c_{i}) \rangle^2 = \{0\}$ in $L/I(H',S') \cong L_{K}(E\backslash (H',S'))$, which contradicts Theorem \ref{LPAbasics}(2). Thus $c_{i} = d_{j}$, and so $g_{j}(d_{j})=g_{j}(c_{i})\in \langle \{c_i^0\}\rangle$, which implies that $\langle \{c_i^0\}\rangle \cap\langle g_{j}(d_{j})\rangle=\langle g_{j}(c_{i})\rangle$. Noting that the cycles $d_{j}$ are distinct, we conclude that
\[
M\cap\bar{J} = \sum_{j\in X}\sum_{i\in Y} [\langle \{c_i^0\}\rangle\cap\langle g_{j}(d_{j})\rangle] = \sum_{i\in Y'}\langle g_{i}(c_{i})\rangle
\]
for some $Y' \subseteq Y$, upon reindexing the $g_{j}$.
\end{proof}

\begin{notation}
Let $I$ be a non-graded ideal in a Leavitt path algebra $L_K(E)$. In the notation of Theorem \ref{arbitrary ideals}(1), $I$ can be expressed as $I=I(H,S)+ \sum_{i\in Y} \langle f_{i}(c_{i})\rangle$ for some cycles $c_{i}$. As mentioned in Remark \ref{Uniqueness of cycles}, the cycles $c_{i}$ are uniquely determined by $I$. For the remainder of this section we shall denote $\{c_{i} : i\in Y\}$ by $\mathrm{Cyc}(I)$. 
\end{notation}

We require one more technical result, which will allow us to rewrite products of (general) ideals in a more convenient way.

\begin{proposition} \label{EqualGradedPart} 
Let $L=L_K(E)$ be a Leavitt path algebra, and let $I, A_{1}, \dots, A_{n}$ be proper ideals of $L$ such that $I=A_{1}\cdots A_{n}$ and $I \not\subseteq \mathrm{gr}(A_{k})$ for all $k \in \{1,\dots,n\}$. Then there exist ideals $B_{1}, \dots, B_{n}$ of $L$ such that $I=B_{1}\cdots B_{n}$, and for each  $k \in \{1,\dots,n\}$ we have $B_{k}\subseteq A_{k}$, $\mathrm{gr}(B_{k})=\mathrm{gr}(I)$, and $\mathrm{Cyc}(B_{k}) = \mathrm{Cyc}(I)$. Moreover, if $A_{k}$ is semiprime for some $k \in \{1, \dots, n\}$, then so is $B_{k}$.
\end{proposition}

\begin{proof}
If $I=A_{1}\cdots A_{n}$ were a graded ideal, then, by Lemma \ref{Product=Intersection}(2), 
\[
I = A_{1} \cap \cdots \cap A_{n} =\mathrm{gr}(A_{1})\cap \cdots \cap \mathrm{gr}(A_{n}),
\] 
contrary to our hypothesis that $I \not\subseteq \mathrm{gr}(A_{1})$. Thus $I$ is necessarily not graded, and so, by Theorem \ref{arbitrary ideals}(1), we can write $I=I(H,S)+\sum_{i\in Y} \langle f_{i}(c_{i})\rangle$, where $Y \neq \emptyset$, each $c_{i}$ is a cycle without exits in $E\backslash(H,S)$, and each $f_{i}(x)\in K[x]$ has a nonzero constant term. Likewise, for each $k \in \{1,\dots,n\}$ we can write $A_{k} = I(H_{k},S_{k})+\sum_{j\in X_{k}} \langle g_{jk}(c_{jk})\rangle$, where each $c_{jk}$ is a cycle without exits in $E\backslash(H_{k},S_{k})$, and each $g_{jk}(x)\in K[x]$ has a nonzero constant term. Before defining the appropriate $B_{k}$, we shall first need to relate the structure of the $A_{k}$ to that of $I$ more closely.

In $L/I(H,S)\cong L_{K} (E\backslash(H,S))$, we have, by Theorem \ref{arbitrary ideals}(2),
\[
\bar{I}=I/I(H,S)=\bigoplus_{i\in Y} \langle f_{i}(c_{i})\rangle\subseteq M= \bigoplus_{i\in Y}\langle \{c_i^0\}\rangle.
\]
Since $I(H,S) \subseteq I \subseteq A_{k}$ for each $k \in \{1,\dots,n\}$, it follows that $\bar{I}=\bar{A}_{1} \cdots \bar{A}_{n}$, where $\bar{A}_{k}=A_{k}/I(H,S)$. Using the fact that $M$ is a graded ideal containing $\bar{I}$ and Lemma \ref{Product=Intersection}(1), we then obtain
\[
\bar{I} =M\bar{I}=M^{n}\bar{A}_{1}\cdots \bar{A}_{n}=M\bar{A}_{1} \cdots M\bar{A}_{n} =(M\cap\bar{A}_{1}) \cdots (M\cap\bar{A}_{n}).
\]
Now, our hypothesis that $I \not\subseteq \mathrm{gr}(A_{k})$ implies that $M\cap \bar{A}_{k} \neq M$, for each $k \in \{1,\dots,n\}$. Hence, by Lemma \ref{intersect-lemma}, we have $M\cap \mathrm{gr}(\bar{A}_{k})=\{0\}$ and $M\cap\bar{A}_{k} = \sum_{i\in Y'_{k}}\langle g_{ik}(c_{i})\rangle$ for some $Y'_{k} \subseteq Y$ and $g_{ik}(x)\in K[x]$, having nonzero constant terms.

We are now ready to define the $B_{k}$. Specifically, for each $k \in \{1,\dots,n\}$ let $B_{k}$ be the ideal of $L$ satisfying $I(H,S)\subseteq B_{k}\subseteq A_{k}$ such that $\bar{B}_{k}=B_{k}/I(H,S)=M\cap\bar{A}_{k}.$
Then
\[
\bar{B}_{1} \cdots \bar{B}_{n}=(M\cap\bar{A}_{1}) \cdots (M\cap\bar{A}_{n})=\bar{I}=\bar
{A}_{1} \cdots \bar{A}_{n}.
\]
Since $I(H,S)\subseteq B_{k}$ for each $k$, it follows that $B_{1}\cdots B_{n} = I$. Also, since $M\cap \mathrm{gr}(\bar{A}_{k})=\{0\}$ and $\bar{B}_{k} = \sum_{i\in Y'_{k}}\langle g_{ik}(c_{i})\rangle$, we conclude that $\mathrm{gr}(B_{k})=I(H,S) =\mathrm{gr}(I)$ and $ \mathrm{Cyc}(B_{k})\subseteq \mathrm{Cyc}(I)$ for each $k$. On the other hand, by Lemma \ref{Product}, $\bigoplus_{i\in Y}\langle f_{i}(c_{i})\rangle=\bar{I}=\bar{B}_{1}\cdots \bar{B}_{n}$ implies that $\mathrm{Cyc}(I)\subseteq \mathrm{Cyc}(B_{k})$ for each $k$. Hence $\mathrm{Cyc}(B_{k})= \mathrm{Cyc}(I)$ for all $k \in \{1,\dots,n\}$.

Finally, if some $A_{k}$ is semiprime, then so is $\bar{A}_{k}$, and hence so is $\bar{B}_{k} = M\cap\bar{A}_{k}$, since $M$ is semiprime, by Theorem \ref{Radical ideals}. It follows that $B_{k}$ is semiprime as well.
\end{proof}

We now turn to our analysis of semiprime ideals in a Leavitt path algebra. Recall that a description of semiprime ideals in terms of their generating sets is given in Theorem \ref{Radical ideals}. We next answer the followup question: when is \textit{every} proper ideal in a Leavitt path algebra semiprime?   

\begin{theorem} \label{Every I radical}
The following are equivalent for any Leavitt path algebra $L=L_K(E)$.
\begin{enumerate}
\item[$(1)$] Every proper ideal of $L$ is semiprime.

\item[$(2)$] Every ideal of $L$ is graded.

\item[$(3)$] For any finite list $P_{1}, \dots, P_{n}$ of prime ideals of $L$, we have $P_{1}\cdots  P_{n} = P_{1}\cap\cdots \cap P_{n}$.

\item[$(4)$] The graph $E$ satisfies Condition (K).
\end{enumerate}
\end{theorem}

\begin{proof}
(2) $\Rightarrow$ (1) This follows from the fact that, by Theorem \ref{Radical ideals}, every proper graded ideal in $L$ is semiprime.

(1) $\Rightarrow$ (3) Suppose that every proper ideal of $L$ is semiprime, and let $P_{1}, \dots, P_{n}$ be prime ideals. As mentioned before, it is a standard fact that every semiprime ideal is the intersection of the prime ideals that contain it, and hence $P_{1} \cdots P_{n} = \bigcap_{i\in Y}Q_{i}$ for some prime ideals $Q_{i}$ of $L$. Now for any $i \in Y$, the ideal $Q_{i}$ being prime implies that $P_{j_i} \subseteq Q_{i}$ for some $j_i \in \{1, \dots, n\}$, and moreover,  $P_{1} \cdots P_{n} \subseteq P_{i}$ for each $i \in \{1, \dots, n\}$. It follows that 
\[
P_{1} \cdots P_{n} = \bigcap_{i\in Y}Q_{i} \cap \bigcap_{i \in Y} P_{j_i} = \bigcap_{i\in Y}(Q_{i} \cap P_{j_i}) = \bigcap_{i\in Y} P_{j_i} \supseteq P_{1}\cap\cdots \cap P_{n} \supseteq P_{1} \cdots P_{n},
\]
and so $P_{1} \cdots P_{n} = P_{1}\cap\cdots \cap P_{n}$.

(3) $\Rightarrow$ (2) Suppose, seeking a contradiction, that (3) holds, and there is a non-graded ideal $I$ in $L$. By Theorem \ref{arbitrary ideals}(1), $I=I(H,S)+ \sum_{i\in Y} \langle f_{i}(c_{i})\rangle$, where $Y$ is nonempty, each $c_{i}$ is a cycle without exits in $E\backslash(H,S)$, and each $f_{i}(x)\in K[x]$ has a nonzero constant term. Now fix an $i\in Y$. Since $c_{i}$ has no exits in $E\backslash(H,S)$, it cannot be the case that a vertex on $c_{i}$ is the base of another cycle in $E$; that is, $c_{i}$ is a cycle without (K). Moreover, letting $H'=\{w\in E^{0} : w\ngeq s(c_{i})\}$, it is easy to see that $H'$ is hereditary and saturated, and that $H \subseteq H'$. It follows that $c_{i}$ has no exits in $E \backslash (H',B_{H'})$, and therefore $P=I(H^{\prime},B_{H'})+\langle p(c_{i})\rangle$ is a prime ideal, for any irreducible polynomial $p(x)\in K[x,x^{-1}]$, by Theorem \ref{specificprime}. Then, by Theorem \ref{arbitrary ideals}(2), in $L/I(H^{\prime}, B_{H^{\prime}}) \cong L_{K} (E\backslash H^{\prime},B_{H^{\prime}})$, we have
\[
\bar{P}=P/I(H^{\prime},B_{H^{\prime}})=\langle p(c_{i})\rangle\subseteq \langle \{c_{i}^{0}\}\rangle.
\] 
Moreover, $\langle \{c_{i}^{0}\}\rangle \cong M_{\Lambda}(K[x,x^{-1}])$ for some index set $\Lambda$, by Lemma \ref{cycle-ideal-lemma}. Then Proposition \ref{Morita equivalence} and Lemma \ref{ProductNOTintersection} imply that $\bar{P}^{2}$ is not an intersection of prime ideals of $\langle \{c_{i}^{0}\}\rangle$. Since $M_{\Lambda}(K[x,x^{-1}])$ is a ring with local units, every ideal (respectively, prime ideal) of $\langle \{c_{i}^{0}\}\rangle$ is an ideal (respectively, prime ideal) of $L/I(H^{\prime},B_{H^{\prime}})$, and $\langle \{c_{i}^{0}\}\rangle\cap Q$ is a prime ideal of $\langle \{c_{i}^{0}\}\rangle$ for any prime ideal $Q$ of $L/I(H^{\prime},B_{H^{\prime}})$. Consequently, $\bar{P}^2$ cannot be an intersection of prime ideals of $L/I(H^{\prime},B_{H^{\prime}})$, and hence $P^{2}$ cannot be an intersection of prime ideals of  $L$. In particular $P^{2} \neq P \cap P$, contradicting (3). Hence every ideal of $L$ must be graded.

(2) $\Leftrightarrow$ (4) See Theorem \ref{LPAbasics}(1).
\end{proof}

Easy ring-theoretic considerations yield numerous examples of Leavitt path algebras which satisfy the conditions of Theorem \ref{Every I radical} but not those of Theorem \ref{Everyone prime}; for instance, $K\oplus K \cong L_K(\bullet \ \bullet)$ is one such.   Here is a more interesting example, one which relies on the graph-theoretic structure of $E$.  

\begin{example}
Let $E$ be the following graph.

\smallskip
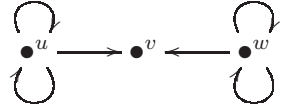
\[ 
\xymatrix{ \bullet^{u} \ar@(ul,ur) \ar@(dr,dl) \ar[r] & \bullet^{v} &  \bullet^{w} \ar[l] \ar@(ul,ur) \ar@(dr,dl)}
\]
\smallskip

\noindent
Then $E$ clearly satisfies Condition (K), and hence every proper ideal of $L_K(E)$ is semiprime, by Theorem \ref{Every I radical}. However, not every proper ideal in this ring is prime, by Theorem \ref{Everyone prime}, since the admissible pairs for $E$ do not form a chain. More specifically, the proper hereditary saturated subsets of $E^{0}$ are $\{v\}$, $\{u,v\}$, and $\{v,w\}$, and these clearly do not form a chain under set inclusion. Since $E$ is row-finite, and $B_H = \emptyset$ for every hereditary saturated $H \subseteq E^{0}$, it follows that the admissible pairs $(H,S)$ do not form a chain in the relevant partial order either. \hfill $\Box$
\end{example}

We are now ready for the key result of the article, a description of the ideals in a Leavitt path algebra which can be written as products of semiprime ideals.  

\begin{theorem} \label{Product of semiprimes} 
Let $L=L_K(E)$ be a Leavitt path algebra, and let $I$ be a proper ideal of $L$. Then the following are equivalent.
\begin{enumerate}
\item[$(1)$] $I$ is a product of semiprime ideals.

\item[$(2)$] $I=J_{1}\cdots J_{n}$  for some semiprime ideals $J_{1}, \dots, J_{n}$ of $L$, where $\mathrm{gr}(J_{j})=\mathrm{gr}(I)$ and $\mathrm{Cyc}(J_{j})=\mathrm{Cyc}(I)$ for all $j \in \{1,\dots, n\}$.

\item[$(3)$] $I=I(H,S)+\sum_{i\in Y} \langle f_{i}(c_{i})\rangle$, where $Y$ is a possibly empty index set, each $c_{i}$ is a cycle without exits in $E\backslash(H,S)$, each $f_{i}(x)\in K[x]$ has a nonzero constant term; and there is a positive integer $n$ such that, for each $i\in Y$, there exist pairwise non-conjugate irreducible polynomials $p_{1}(x),\dots, p_{k}(x) \in K[x]$ and integers $1 \leq m_{1},\dots, m_{k} \leq n$ satisfying $f_{i}(x)=p_{1}^{m_{1}}(x)\cdots p_{k}^{m_{k}}(x)$.
\end{enumerate}
\end{theorem}

\begin{proof}
(1) $\Rightarrow$ (2) If $I$ is graded, then, by Theorem \ref{Radical ideals}, $I$ is semiprime, and so (2) holds vacuously. Thus, we may assume that $I$ is not graded, and that $I = A_{1} \cdots A_{n}$ for some semiprime ideals $A_{1}, \dots, A_{n}$. 

Next, suppose that $I \subseteq \mathrm{gr}(A_{j})$ for all $j \in \{1, \dots, n\}$, and let $B = \bigcap_{j=1}^n \mathrm{gr}(A_{j})$. Since $I \subseteq B$, and $B$ is graded, by Lemma \ref{Product=Intersection}(1), we have 
\[
I = B^{n}I = BA_{1} \cdots BA_{n} = (B\cap A_{1}) \cdots (B \cap A_{n}) = B^{n} = B,
\]
which contradicts $I$ being non-graded. Thus, $I \not\subseteq \mathrm{gr}(A_{j})$ for at least one $j$. Upon reindexing, we may assume that $I \not\subseteq \mathrm{gr}(A_{j})$ for $j \in \{1, \dots, k\}$, and $I \subseteq \mathrm{gr}(A_{j})$ for $j \in \{k+1, \dots, n\}$. Also, let $B = \bigcap_{j=k+1}^n \mathrm{gr}(A_{j})$ (with $B$ taken to be $L$, if $k=n$), and let $A'_{j} = B \cap A_{j}$ for each $j \in \{1, \dots, k\}$. Then $I \not\subseteq \mathrm{gr}(A'_{j})$, and, as an intersection of semiprime ideals, $A'_{j}$ is semiprime, for each $j$. Moreover, as before,
\[
I = BA_{1} \cdots BA_{n} =  (B\cap A_{1}) \cdots (B \cap A_{n}) = A'_{1} \cdots A'_{k} \cdot B^{n-k} = A'_{1} \cdots A'_{k}.
\]
The desired conclusion now follows from Proposition \ref{EqualGradedPart}.

(2) $\Rightarrow$ (3) Let $J_{1}, \dots, J_{n}$ be as in (2). By Theorem \ref{arbitrary ideals}(1) and Theorem \ref{Radical ideals}, we can then write $I=I(H,S)+\sum_{i\in Y} \langle f_{i}(c_{i})\rangle$, and for each $k \in \{1,\dots,n\}$, $J_{k}=I(H,S)+\sum_{i\in Y}
\langle g_{ik}(c_{i})\rangle$, where each $c_{i}$ is a cycle without exits in $E \backslash (H,S)$, each $f_{i}(x), g_{ik}(x) \in K[x]$ has a nonzero constant term, and each $g_{ik}(x)$ is square-free. In $L/I(H,S) \cong L_{K}(E\backslash(H,S))$, we have $\bar{I}=\bar{J}_{1}\cdots \bar{J}_{n}$, where $\bar{I}=I/I(H,S)$ and $\bar{J}_{k}=J_{k}/I(H,S)$ for each $k \in \{1,\dots,n\}$. So, by Theorem \ref{arbitrary ideals}(2),
\[
\bar{I}= \bigoplus_{i\in Y} \langle f_{i}(c_{i})\rangle=\Big[\bigoplus_{i\in Y} \langle g_{i1}(c_{i})\rangle\Big]\cdots \Big[\bigoplus_{i\in Y} \langle g_{in}(c_{i})\rangle\Big].
\]
Equating the terms corresponding to each $c_{i}$ on the two sides of the above equation, by Lemmas \ref{Product} and \ref{cycle-ideal-lemma}, we conclude that for each $i\in Y$, 
\[
\langle f_{i}(c_{i})\rangle=\langle g_{i1}(c_{i})\cdots g_{in}(c_{i})\rangle\subseteq \langle \{c_{i}^{0}\}\rangle\cong M_{\Lambda_{i}}(K[x,x^{-1}]).
\]
By Proposition \ref{Morita equivalence} and Remark \ref{cycle-ideal-remark}, for each $i\in Y$ we then have $\langle f_{i}(x)\rangle=\langle g_{i1}(x) \cdots g_{in}(x)\rangle$ in $K[x,x^{-1}]$. Since each of the polynomials involved has a nonzero constant term, it follows that $f_{i}(x)h(x) = g_{i1}(x) \cdots g_{in}(x)$ for some $h(x) \in K[x]$. Since each $g_{ik}(x)$ is square-free in $K[x]$, we conclude that if $f_{i}(x)$ is divisible by $p^{m}(x)$, for some irreducible $p(x) \in K[x]$ and positive integer $m$, then $m\leq n$, proving (3).

(3) $\Rightarrow$ (1) Suppose that the ideal $I = I(H,S)+\sum_{i\in Y} \langle f_{i}(c_{i})\rangle$ satisfies the conditions in (3). If $Y = \emptyset$, then $I$ is graded, and is hence itself a semiprime ideal, by Theorem \ref{Radical ideals}. So we may assume that $Y \neq \emptyset$, and proceed by induction on $n$. 

Suppose that $n=1$. Then, for each $i\in Y$, $f_{i}(x)=p_{1}(x)\cdots p_{k}(x)$ for some pairwise non-conjugate irreducible polynomials $p_{1}(x),\dots, p_{k}(x) \in K[x]$, and hence, once again, $I$ itself is a semiprime ideal, by Theorem \ref{Radical ideals}. 

Now suppose that $n>1$ and that (1) holds for all $I$ satisfying (3), with the bound being
$n-1$. Let us write $Y = Y_{1} \cup Y_{2}$, such that the following conditions are satisfied.
\begin{enumerate}
\item[(i)] For each $i\in Y_{1}$, $f_{i}(x)=p_{i1}^{n}(x)\cdots p_{ir_{i}}^{n}(x)q_{i1}^{k_{i1}}(x)\cdots q_{is_{i}}^{k_{is_{i}}} (x),$ where $1\leq k_{ij}\leq n-1$ for each $j$, and the $p_{ij}(x), q_{ij}(x) \in K[x]$ are pairwise non-conjugate and irreducible.

\item[(ii)] For each $i\in Y_{2}$, $f_{i}(x)=q_{i1}^{k_{i1}}(x)\cdots q_{is_{i}}^{k_{is_{i}}}(x)$, where $1\leq k_{ij}\leq n-1$ for each $j$, and the $q_{ij}(x) \in K[x]$ are pairwise non-conjugate and irreducible.
\end{enumerate}
For each $i\in Y_{1}$, write $f_{i}(x)=f_{i(1)}(x)f_{i(2)}(x)$, where $f_{i(1)}(x)=p_{i1}(x)\cdots p_{ir_{i}}(x)$ and 
\[
f_{i(2)}(x)=p_{i1}^{n-1}(x)\cdots p_{ir_{i}}^{n-1}(x)q_{i1}^{k_{i1}}(x)\cdots q_{is_{i}}^{k_{is_{i}}}(x).
\]
Finally, by Theorem \ref{arbitrary ideals}(2), we can find ideals $J_{1}$ and $J_{2}$ of $L$ containing $I(H,S)$, such that in $L/I(H,S)\cong L_{K}(E\backslash(H,S))$,
\[
\bar{J}_{1}=J_{1}/I(H,S)= \bigoplus_{i\in Y_{1}} \langle f_{i(1)}(c_{i})\rangle\oplus
\bigoplus_{i\in Y_{2}} \langle\{c^{0}_{i}\}\rangle
\]
and
\[
\bar{J}_{2}=J_{2}/I(H,S)= \bigoplus_{i\in Y_{1}} \langle f_{i(2)}(c_{i})\rangle\oplus \bigoplus_{i\in Y_{2}} \langle f_{i}(c_{i})\rangle.
\]
Then, by the inductive hypothesis,
\[
J_{2}=I(H,S)+ \sum_{i\in Y_{1}} \langle f_{i(2)}(c_{i})\rangle+ \sum_{i\in Y_{2}} \langle f_{i}(c_{i})\rangle
\]
is a product of semiprime ideals. We shall complete the proof by showing that $I = J_{1}J_{2}$ and that $J_{1}$ is semiprime.

Noting that $\langle\{c^{0}_{i}\}\rangle$ is a graded ideal for each $i \in Y_{2}$, by Lemma \ref{Product} and Lemma \ref{Product=Intersection}(1), in $L/I(H,S) \cong L_{K}(E\backslash(H,S))$ we then have
\begin{align*}
\bar{J}_{1}\bar{J}_{2}  &  = \Big[\bigoplus_{i\in Y_{1}}\langle f_{i(1)}(c_{i})\rangle\oplus
\bigoplus_{i\in Y_{2}} \langle\{c^{0}_{i}\}\rangle \Big]\Big[\bigoplus_{i\in Y_{1}}\langle f_{i(2)}(c_{i})\rangle \oplus \bigoplus_{i\in Y_{2}}\langle f_{i}(c_{i})\rangle\Big]\\
&  = \bigoplus_{i\in Y_{1}} (\langle f_{i(1)}(c_{i})\rangle\langle f_{i(2)}(c_{i})\rangle) \oplus \bigoplus_{i\in Y_{2}} \langle\{c^{0}_{i}\}\rangle \langle f_{i}(c_{i})\rangle\\
&  = \bigoplus_{i\in Y_{1}} \langle f_{i}(c_{i})\rangle \oplus
\bigoplus_{i\in Y_{2}} \langle f_{i}(c_{i})\rangle = \bigoplus_{i\in Y} \langle f_{i}(c_{i})\rangle.
\end{align*}
Consequently 
\[
J_{1}J_{2}=I(H,S)+ \sum_{i\in Y} \langle f_{i}(c_{i})\rangle=I,
\] 
and so it remains to show that $J_{1}$ is semiprime.

Since the preimage of $\bigoplus_{i\in Y_{2}} \langle\{c^{0}_{i}\}\rangle \subseteq L/I(H,S)$ in $L$ under the natural projection is a graded ideal, there is some admissible pair $(H_{1},S_{1})$ such that $J_{1}= I(H_{1},S_{1})+ \sum_{i\in Y_{1}} \langle f_{i(1)}(c_{i})\rangle.$ By the correspondence between admissible pairs and graded ideals of $L$, we have $(H,S) \leq (H_{1},S_{1})$. Thus, for each $i \in  Y_{1}$, since $c_{i}$ is a cycle without exits in $E\backslash(H,S)$, it must also be without exits in $E\backslash(H_{1},S_{1})$. Since each $f_{i(1)}(x)$ is square-free, we conclude, by Theorem \ref{Radical ideals}, that $J_{1}$ is a semiprime ideal, as desired.
\end{proof}

With Theorem \ref{Product of semiprimes} in hand, for appropriate contrast we provide an example of an ideal in a Leavitt path algebra which cannot be written as a product of semiprime ideals.   

\begin{example}\label{notproductofsemiprimes}
Let $E$ be the following row-finite graph.

\smallskip
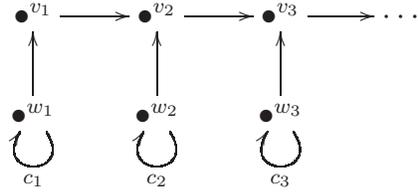
\[ 
\xymatrix{   \bullet^{v_1} \ar[r] & \bullet^{v_2} \ar[r] & \bullet^{v_3} \ar[r] & \cdots \\
  \bullet^{w_{1}} \ar[u] \ar@(dr,dl)^{c_1}  &  \bullet^{w_{2}} \ar[u]  \ar@(dr,dl)^{c_2}  &  \bullet^{w_{3}}  \ar[u] \ar@(dr,dl)^{c_3}  &  }
\]
\smallskip

\noindent
The set $H = \{v_i : i\in \mathbb{Z}^+\}$ is a hereditary saturated subset of $E^0$, and each $c_i$ is  a cycle without exits in $E^0 \setminus H$.     
 
For each $i\in \mathbb{Z}^+$ consider the polynomial $f_i(x) = (1 +  x)^i$ in $K[x]$.  We now form the ideal 
\[
I = I(H) + \sum_{i \in \mathbb{Z}^+} \langle f_i(c_i) \rangle = \langle \{ v_i : i\in \mathbb{Z}^+  \} \rangle + \sum_{i \in \mathbb{Z}^+} \langle  (w_i + c_i)^i \rangle
\] 
of $L_K(E)$. Then for each $i\in \mathbb{Z}^+$ the only monic irreducible factor of $f_i(x)$ is $p(x) = 1+x$. So there does not exist a positive integer $n$ for which each $f_i(x)$ is the product of $\leq n$ copies of $p(x)$.  Thus $I$ is not a product of semiprime ideals in $L_K(E)$. \hfill $\Box$ 
\end{example}  

Theorem \ref{Product of semiprimes} has the following interesting consequence.

\begin{corollary} \label{RadFactorization copy(1)} 
Let $L$ be a Leavitt path algebra. If $I$ is a proper ideal of $L$ such that $I/\mathrm{gr}(I)$ is finitely generated, then $I$ is a product of semiprime ideals of $L$.

In particular, if $L$ is two-sided Noetherian, then every proper ideal of $L$ is a product of semiprime ideals.
\end{corollary}

\begin{proof}
By Theorem \ref{arbitrary ideals}(1), we can write $I=I(H,S)+ \sum_{i\in Y} \langle f_{i}(c_{i})\rangle$, where each $c_{i}$ is a cycle without exits in $E\backslash(H,S)$, and each $f_{i}(x)\in K[x]$ has a nonzero constant term. Then $I/I(H,S) \cong  \bigoplus_{i\in Y}\langle f_{i}(c_{i}) \rangle$, by Theorem \ref{arbitrary ideals}(2). Our hypothesis then implies that $Y$ is finite. Hence $I$ satisfies condition (3) of Theorem \ref{Product of semiprimes}, and is therefore a product of semiprime ideals.

If $L$ is Noetherian and $I$ is a proper ideal, then necessarily $I$, and hence also $I/\mathrm{gr}(I)$, is finitely generated. Thus, in this situation $I$ is product of semiprime ideals.
\end{proof}

It is shown in \cite[Corollary 16]{C} that if $E$ is finite (i.e., $E^{0}$ and $E^{1}$ are both finite), then $L_{K}(E)$ is Noetherian. Hence Corollary \ref{RadFactorization copy(1)} implies that every proper ideal of $L_{K}(E)$ is a product of semiprime ideals whenever $E$ is finite.

Our next result is the expected one, where we answer the question: for which Leavitt path algebras is it the case that every proper ideal is a product of semiprime ideals?

\begin{theorem} \label{Every I is prod. semiprimes} 
The following are equivalent for any Leavitt path algebra $L=L_K(E)$.
\begin{enumerate}
\item[$(1)$] Every proper ideal of $L$ is a product of semiprime ideals.

\item[$(2)$] For every ideal $I$ of $L$, $I/\mathrm{gr}(I)$ is finitely generated.

\item[$(3)$] Given any hereditary saturated subset $H$ of $E^0$, there are only finitely many cycles $c$ in $E$ with the property that $\{c^{0}\} \cap H = \emptyset$ and $r(e) \in H$ for all exits $e$ of $c$.  
\end{enumerate}
\end{theorem}

\begin{proof}
(1) $\Rightarrow$ (2) We shall prove the contrapositive. Suppose that there is an ideal $I$ in $L$ for which $I/\mathrm{gr}(I)$ is not finitely generated. Then, by Theorem \ref{arbitrary ideals}(1), we can write $I=I(H,S)+ \sum_{i\in Y} \langle f_{i}(c_{i})\rangle$, where $Y$ is an infinite set, each $c_{i}$ is a cycle without exits in $E\backslash(H,S)$, and each $f_{i}(x)\in K[x]$. Let $N$ be a countably infinite subset of $Y$, which we identify with the set of positive integers. Then, by Theorem \ref{Product of semiprimes}, the ideal $A=I(H,S)+ \sum_{i\in N} \langle (s(c_{i}) + c_{i})^{i}\rangle$ is not a product of semiprime ideals.

(2) $\Rightarrow$ (3) As before, we shall prove the contrapositive. Suppose that there is a hereditary saturated subset $H$ of $E^0$, and there is an infinite set of cycles $\{c_{i} : i\in Y\}$ in $E$, with each $c_{i}$ having the property that $\{c_{i}^{0}\} \cap H = \emptyset$ and $r(c_{i}) \in H$ for all exits $e$ of $c_{i}$. Clearly each $c_{i}$ is a cycle without exits in $E \setminus (H, B_{H})$. Now, let $I = I(H,B_{H}) + \sum_{i\in Y} \langle s(c_{i}) + c_{i}\rangle$. Then $I/I(H,B_{H}) =  \bigoplus_{i\in Y}\langle s(c_{i}) + c_{i} \rangle$, by Theorem \ref{arbitrary ideals}(2), and hence $I/\mathrm{gr}(I) = I/I(H,B_{H})$ is not finitely generated.

(3) $\Rightarrow$ (1) Suppose that (3) holds, and let $I$ be a proper ideal of $L$. By Theorem \ref{arbitrary ideals}(1), we can write $I=I(H,S)+ \sum_{i\in Y} \langle f_{i}(c_{i})\rangle$, where each $c_{i}$ is a cycle without exits in $E \backslash (H,S)$, and each $f_{i}(x)\in K[x]$ has a nonzero constant term. Our hypothesis implies that there can be only finitely many cycles without exits in $E\backslash(H,S)$, and hence $Y$ must be finite. It follows that $I$ is a product of semiprime ideals, by Corollary \ref{RadFactorization copy(1)} (or Theorem \ref{Product of semiprimes}).
\end{proof}

Let $E$ be the one-vertex-one-loop graph, as in Example \ref{primenotprimaryremark}. Then using ring-theoretic considerations, it is easy to see that $K[x,x^{-1}] \cong L_K(E)$ satisfies the conditions of  Theorem \ref{Every I is prod. semiprimes} but not those of Theorem \ref{Every I radical}. Here is another example, one for which the justification is most natural from the graph-theoretic point of view.  


\begin{example}
Let $E$ be the following graph.

\smallskip
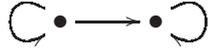
\[ 
\xymatrix{ \bullet \ar@(ul,dl) \ar[r] & \bullet \ar@(ur,dr)}
\]
\smallskip

\noindent
Since there are only finitely many cycles in $E$, this graph satisfies condition (3) in Theorem \ref{Every I is prod. semiprimes} trivially. Hence every proper ideal of $L_K(E)$ is a product of semiprime ideals. However, since $E$ certainly does not satisfy Condition (K), not every proper ideal of $L_K(E)$ is semiprime, by Theorem \ref{Every I radical}.  \hfill $\Box$
\end{example}

We conclude the article with two remarks.  First, while the uniqueness of irredundant prime factorization of ideals in Leavitt path algebras was established in \cite{EKR}, there is no analogous uniqueness result for semiprime factorization of ideals. An easy example suffices: in the graph $E$ consisting of four vertices $\{u,v,w,x\}$ and no edges, the distinct ideals $\langle u \rangle$, $\langle v \rangle$, $\langle w \rangle$, and $\langle x \rangle$ are each semiprime (as each is graded), but $\langle u \rangle \langle v \rangle = \langle w \rangle \langle x \rangle$ (as each product is the zero ideal). 

Second, we note some connecting implications among our main results. Of course every proper ideal being prime (Theorem \ref{Everyone prime}) trivially implies that every proper ideal is a product of primes (Theorem \ref{everyone-a-product-of-primes}), which then implies that every proper ideal is a product of semiprimes (Theorem \ref{Every I is prod. semiprimes}). As well, every proper ideal being prime (Theorem \ref{Everyone prime}) implies that every proper ideal is semiprime (Theorem \ref{Every I radical}), which then too (trivially) implies that every proper ideal is a product of semiprimes (Theorem \ref{Every I is prod. semiprimes}). However, the conditions given in Theorem \ref{everyone-a-product-of-primes} (every proper ideal is a product of primes) and Theorem \ref{Every I radical} (every proper ideal is semiprime) are independent. For example, if $E$ is the graph of Example \ref{primenotprimaryremark}, then $K[x,x^{-1}] \cong L_K(E)$
satisfies the conditions of  Theorem \ref{everyone-a-product-of-primes}, but not of Theorem \ref{Every I radical}. Conversely, if $E$ is the graph of Example \ref{exampleidealnotproductprimes}, then $L_K(E)$ satisfies the conditions of Theorem \ref{Every I radical}, but not of Theorem \ref{everyone-a-product-of-primes}.  

\subsection*{Acknowledgement}

We are grateful to the referee for a careful reading of the manuscript.

\vspace{.1in}

\noindent
Department of Mathematics, University of Colorado, Colorado Springs, CO, 80918, USA \newline
\noindent {\href{mailto:abrams@math.uccs.edu}{abrams@math.uccs.edu}},  {\href{mailto:zmesyan@uccs.edu}{zmesyan@uccs.edu}},  {\href{mailto:krangasw@uccs.edu}{krangasw@uccs.edu}}

\end{document}